\documentclass[12pt]{article}
 
\usepackage{amsmath}
\usepackage{amsfonts}
\usepackage{amsthm}
\usepackage{graphicx}
\newtheorem{theorem}{Theorem}[section]
\newtheorem{proposition}[theorem]{Proposition}
\numberwithin{equation}{section}

\begin{document}
\title{Multiple Askey-Wilson polynomials and related basic hypergeometric multiple orthogonal polynomials}
\author{Jean Paul Nuwacu and Walter Van Assche \\
Universit\'e du Burundi and KU Leuven, Belgium}
\date{\today}
\maketitle

\begin{abstract}
We first show how one can obtain Al-Salam--Chihara polynomials, continuous dual $q$-Hahn polynomials, and Askey--Wilson polynomials from the little $q$-Laguerre and the little $q$-Jacobi polynomials by using special transformations. This procedure is then extended to obtain multiple Askey--Wilson, multiple continuous dual
$q$-Hahn, and multiple Al-Salam--Chihara polynomials from the multiple little $q$-Laguerre and the multiple
little $q$-Jacobi polynomials. 
\end{abstract}

\section{Introduction}   \label{sec:intro}
In this paper we will first show in Section \ref{sec:2} how some families of basic hypergeometric polynomials are related by a linear transformation.
This transformation is a $q$-analogue of the Fourier--Jacobi transform that maps Jacobi polynomials to Wilson polynomials
\cite{Koornwinder}. We will consider three families of basic hypergeometric polynomials: the Al-Salam--Chihara polynomials,
the continuous dual $q$-Hahn polynomials and the Askey--Wilson polynomials, and show how they can be obtained by a linear transformation from the little $q$-Laguerre and the little $q$-Jacobi polynomials. We then extend this
procedure to multiple orthogonal polynomials in Section \ref{sec:3}.
 We will first recall multiple little $q$-Laguerre polynomials in Section \ref{sec:3.1}
and multiple little $q$-Jacobi polynomials in Section \ref{sec:3.2} and then apply the linear transformations to obtain
multiple Al-Salam--Chihara polynomials (Section \ref{sec:3.3}), multiple continuous dual $q$-Hahn polynomials (Section \ref{sec:3.4}) and
finally multiple Askey-Wilson polynomials (Section \ref{sec:3.5}).

\subsection{Basic hypergeometric orthogonal polynomials}
Al-Salam--Chihara polynomials $p_n(x;a,b|q)$ satisfy the orthogonality
\[   \int_{-1}^1 p_n(x;a,b|q) p_m(x;a,b|q) w(\theta;a,b|q) \frac{dx}{\sqrt{1-x^2}} = 0, \qquad n \neq m, \]
where $x = \cos \theta$ and
\begin{equation}   \label{AlSalChiw}
    w(\theta;a,b|q) = \frac{(e^{2i\theta},e^{-2i\theta};q)_\infty}
          {(ae^{i\theta},ae^{-i\theta},be^{i\theta},be^{-i\theta};q)_\infty}, 
\end{equation}
with parameters $a,b$ satisfying $|a|,|b| < 1$, see \cite[\S 14.8]{KLS}, \cite[\S 15.1]{Ismail}. They are given by
\begin{equation}   \label{AlSalChi}
   p_n(x;a,b|q) = {}_3\phi_2 \left( \left.\begin{array}{c} q^{-n}, ae^{i\theta}, ae^{-i\theta} \\
                                                        ab, 0 \end{array} \right| q, q \right) .  
\end{equation}

Continuous dual $q$-Hahn polynomials $p_n(x;a,b,c|q)$ satisfy the orthogonality relations
\[  \int_{-1}^1 p_n(x;a,b,c|q) p_m(x;a,b,c|q) w(\theta;a,b,c|q) \frac{dx}{\sqrt{1-x^2}} = 0, \qquad n \neq m, \]
where $x = \cos \theta$ and
\begin{equation}   \label{qHahnw}
    w(\theta;a,b,c|q) = \frac{(e^{2i\theta},e^{-2i\theta};q)_\infty}
          {(ae^{i\theta},ae^{-i\theta},be^{i\theta},be^{-i\theta},ce^{i\theta},ce^{-i\theta};q)_\infty}, 
\end{equation}
with parameters $a,b,c$ satisfying $|a|,|b|,|c| < 1$. They have the basic hypergeometric expression
\begin{equation}  \label{dualqHahn}
  p_n(x;a,b,c|q) = {}_3\phi_2 \left( \left. \begin{array}{c} q^{-n}, ae^{i\theta}, ae^{-i\theta} \\ ab,ac \end{array}
               \right| q,q \right). 
\end{equation}

Askey--Wilson polynomials $p_n(x;a,b,c,d|q)$ satisfy the orthogonality relations
\[   \int_{-1}^1 p_n(x;a,b,c,d|q) p_m(x;a,b,c,d|q) w(\theta;a,b,c,d|q) \frac{dx}{\sqrt{1-x^2}} = 0, \qquad n \neq m, \]
where $x = \cos \theta$ and
\begin{equation}  \label{AWw}
    w(\theta;a,b,c,d|q) = \frac{(e^{2i\theta},e^{-2i\theta};q)_\infty}
          {(ae^{i\theta},ae^{-i\theta},be^{i\theta},be^{-i\theta},ce^{i\theta},ce^{-i\theta},
de^{i\theta},de^{-i\theta};q)_\infty}, 
\end{equation}
with parameters $a,b,c,d$ satisfying $|a|,|b|,|c|,|d| < 1$, \cite[\S 14.1]{KLS}, \cite[\S 15.2]{Ismail}. They are given by
\begin{equation}   \label{AW}
   p_n(x;a,b,c,d|q) = {}_4\phi_3 \left( \left.\begin{array}{c} q^{-n}, abcdq^{n-1}, ae^{i\theta}, ae^{-i\theta} \\
                                                        ab, ac, ad \end{array} \right| q, q \right) .  
\end{equation}
These three families are connected and are in fact all Askey--Wilson polynomials, since 
\begin{itemize}
  \item $p_n(x;a,b,c,0|q) =  p_n(x;a,b,c|q)$,
  \item $p_n(x;a,b,0,0|q) = p_n(x;a,b|q) = p_n(x;a,b,0|q)$.
  \item $a^n H_n(x;a|q) = p_n(x;a,0,0,0|q)$. 
\end{itemize}  
The latter polynomials $H_n(x;a|q)$ are the continuous big $q$-Hermite polynomials which we will use in Section \ref{sec:qHahn}. 
We have normalized the polynomials so that $p_n(0;a,b,c,d|q)=p_n(0;a,b,c|q)=p_n(0;a,b|q)=1$.
Observe that
\begin{equation}  \label{Ak}
  A_k(x) = (ae^{i\theta},ae^{-i\theta};q)_k = \prod_{j=0}^{k-1} (1+a^2q^{2j}-2aq^j x)  
\end{equation}
is a polynomial of degree $k$, so the polynomials \eqref{AlSalChi}, \eqref{dualqHahn}, and \eqref{AW} are expressed as a linear combination of these polynomials $A_k(x)$. 

Two other families of basic hypergeometric orthogonal polynomials that we will encounter are orthogonal with respect to a discrete measure supported
on the $q$-lattice $\{ q^k, k =0,1,2,\ldots\}$. They are the little $q$-Laguerre polynomials $q_n(x;a|q)$ \cite[\S 14.20]{KLS} for which $0 < aq < 1$ and
\begin{equation}  \label{qLagorth}
  \sum_{k=0}^\infty  q_n(q^k;a|q) q_m(q^k;a|q) \frac{(aq)^k}{(q;q)_k} = 0, \qquad m \neq n,
\end{equation}
and the little $q$-Jacobi polynomials $q_n(x;a,b|q)$ \cite[\S 14.12]{KLS}, for which $0 < aq < 1$, $bq < 1$ and
\begin{equation}   \label{qJacorth}
  \sum_{k=0}^\infty q_n(q^k;a,b|q) q_m(q^k;a,b|q) \frac{(aq)^k (bq;q)_k}{(q;q)_k} = 0, \qquad m \neq n.  
\end{equation}
They are given by 
\begin{equation}   \label{qLagsum}
   q_n(x;a|q) = {}_2\phi_1 \left( \left. \begin{array}{c} q^{-n} , 0 \\ aq  \end{array} \right| q, qx \right)
     = \sum_{k=0}^n \frac{(q^{-n};q)_k}{(aq;q)_k(q;q)_k} (qx)^k,
\end{equation}
and
\begin{equation}  \label{qJacsum}
 q_n(x;a,b|q) = {}_2\phi_1 \left( \left. \begin{array}{c} q^{-n} , abq^{n+1} \\ aq  \end{array} \right| q, qx \right)
     = \sum_{k=0}^n \frac{(q^{-n};q)_k(abq^{n+1};q)_k}{(aq;q)_k(q;q)_k} (qx)^k.
\end{equation}
Here we used the normalization $q_n(0;a|q)=1=q_n(0;a,b|q)$.

\subsection{Multiple orthogonal polynomials}
Multiple orthogonal polynomials are polynomials in one variable that have orthogonality conditions with respect to 
several measures. There are two types of multiple orthogonal polynomials, but in this paper we only consider type II
multiple orthogonal polynomials. Let $r \geq 1$ be a positive integer and $(\mu_1,\ldots,\mu_r)$ positive measures
on the real line for which all the moments exist. We will use multi-indices $\vec{n}=(n_1,n_2,\ldots,n_r) \in \mathbb{N}^r$ and denote their size by $|\vec{n}| = n_1+n_2+\cdots+n_r$. Type II multiple orthogonal
polynomials $P_{\vec{n}}(x)$ for the multi-index $\vec{n}$ are monic polynomials of degree $|\vec{n}|$ that satisfy
the orthogonality conditions
\[   \int P_{\vec{n}}(x) x^k\, d\mu_j(x) = 0, \qquad 0 \leq k \leq n_j-1, \]
for $1 \leq j \leq r$. This gives a system of $|\vec{n}|$ homogeneous equations for the $|\vec{n}|$ unknown coefficients of $P_{\vec{n}}$. If the solution exists and if it is unique, then we say that $\vec{n}$ is a normal index. See \cite[Ch. 23]{Ismail}, \cite[\S 4.3]{NikiSor}, \cite{Apt}, \cite{MarVA} for a background on multiple orthogonal polynomials.

During the past few decades, various examples of multiple orthogonal polynomials with classical weights have been worked out. 
Often one can take the orthogonality measures for classical orthogonal polynomials and by allowing $r$ different parameters one gets $r$ measures with respect to which one can look for the corresponding multiple orthogonal polynomials, see, e.g., \cite{AptBrVA, ArCVA, VAC}. Some of these `classical' multiple orthogonal polynomials play an important role in applications, e.g., multiple Hermite polynomials and multiple Laguerre polynomials are used in the analysis of random matrices \cite{BlKuijl1, BlKuijl2, Kuijl1} or special determinantal processes \cite{Kuijl2}, multiple Jacobi polynomials and
multiple little $q$-Jacobi polynomials are used in irrationality proofs \cite{PoVA2, SmetVA, SmetVA2}, multiple Charlier and multiple Meixner polynomials
are used to describe  non-Hermitian oscillator Hamiltonians \cite{MTVZ, MVZ, FrVA}, and in general multiple orthogonal polynomials they are useful in the analysis of multidimensional Schr\"odinger equations and the multidimensional Toda lattice \cite{AptDMVA, AptDVA}.

Beckermann et al. \cite{BCVA} worked out the most general family of classical multiple orthogonal polynomials by giving the multiple Wilson polynomials.
These Wilson polynomials are on top of the Askey table \cite[p.~183]{KLS} and from this family one can move to other families of classical multiple orthogonal polynomials by taking limits. They used a transformation (the Fourier-Jacobi transform) that maps Jacobi polynomials to Wilson polynomials (Koornwinder 
\cite{Koornwinder}) and showed that this transform allows to generate multiple Wilson polynomials from certain multiple Jacobi polynomials
(the Jacobi-Pi\~neiro polynomials). In this paper we will look at the $q$-analogue of the Askey table \cite[p.~413]{KLS}. Some multiple $q$-orthogonal polynomials have already been obtained, such as the multiple little $q$-Jacobi polynomials \cite{PoVA}, multiple $q$-Charlier polynomials \cite{ArRaAb} and multiple $q$-Hahn polynomials \cite{Arvesu}. On top of the $q$-analogue of the Askey table are the Askey--Wilson polynomials and the $q$-Racah polynomials.
In this paper we will obtain multiple Askey--Wilson polynomials (Section \ref{sec:3.5}) by use of a linear transformation that maps little $q$-Jacobi polynomials to Askey--Wilson polynomials. We will also obtain multiple continuous dual $q$-Hahn polynomials (Section \ref{sec:3.4}) and multiple Al-Salam--Chihara polynomials (Section \ref{sec:3.3}) using the transform that maps little $q$-Laguerre polynomials to continuous dual $q$-Hahn polynomials and Al-Salam--Chihara polynomials. to achieve this, we first work out the multiple little $q$-Laguerre polynomials in Section \ref{sec:3.1}
and the multiple little $q$-Jacobi polynomials in Section \ref{sec:3.2}.

\section{A mapping between basic hypergeometric polynomials}  \label{sec:2}
The Al-Salam--Chihara polynomials and the Askey--Wilson polynomials are most naturally expressed in the basis
$\{ A_k(x); k=0,1,2,\ldots \}$ of polynomials, given in \eqref{Ak}. Let us also consider the polynomials
\[   B_k(x) = (be^{i\theta},be^{-i\theta};q)_k = \prod_{j=0}^{k-1} (1+b^2q^{2j}-2bq^j x), \]
then the orthogonality of the Al-Salam--Chihara polynomials is equivalent to
\[   \frac{1}{2\pi} \int_0^\pi p_n(x;a,b|q) (be^{i\theta},be^{-i\theta};q)_j w(\theta;a,b|q)\, d\theta = 0, \qquad
         0 \leq j \leq n-1,  \]
and the orthogonality of the Askey--Wilson polynomials is
\[   \frac{1}{2\pi} \int_0^\pi p_n(x;a,b,c,d|q) (be^{i\theta},be^{-i\theta};q)_j w(\theta;a,b,c,d|q)\, d\theta = 0, \qquad
         0 \leq j \leq n-1.  \]
We can express these polynomials, up to a multiplicative factor, as a determinant:
\begin{equation}  \label{det}
   p_n(x) = C_n \det \begin{pmatrix}  m_{0,0} & m_{0,1} & m_{0,2} & \cdots & m_{0,n}  \\
                                        m_{1,0} & m_{1,1} & m_{1,2} & \cdots & m_{1,n}  \\
                                        m_{2,0} & m_{2,1} & m_{2,2} & \cdots & m_{2,n}  \\
                                        \vdots & \vdots & \vdots & \cdots & \vdots \\
                                        m_{n-1,0} & m_{n-1,1} & m_{n-1,2} & \cdots & m_{n-1,n} \\
                                        A_0(x) & A_1(x) & A_2(x) & \cdots & A_n(x)   \end{pmatrix}  ,
\end{equation}
where the $m_{k,j}$ are modified moments
\begin{eqnarray*}
     m_{k,j}  &=& \frac{1}{2\pi} \int_{-1}^{1} A_j(x) B_k(x) w(\theta) \frac{dx}{\sqrt{1-x^2}} \\
   &=&
   \frac{1}{2\pi} \int_0^\pi (a^{i\theta},ae^{-i\theta};q)_j (be^{i\theta},be^{-i\theta};q)_k w(\theta)\, d\theta .
\end{eqnarray*} 
Indeed, if we integrate then
\[  \frac{1}{2\pi}  \int_{-1}^1 p_n(x) B_k(x) w(\theta)\, \frac{dx}{\sqrt{1-x^2}}
      = C_n \det \begin{pmatrix}  m_{0,0} & m_{0,1} & m_{0,2} & \cdots & m_{0,n}  \\
                                        m_{1,0} & m_{1,1} & m_{1,2} & \cdots & m_{1,n}  \\
                                        m_{2,0} & m_{2,1} & m_{2,2} & \cdots & m_{2,n}  \\
                                        \vdots & \vdots & \vdots & \cdots & \vdots \\
                                        m_{n-1,0} & m_{n-1,1} & m_{n-1,2} & \cdots & m_{n-1,n} \\
                                        m_{k,0} & m_{k,1} & m_{k,2} & \cdots & m_{k,n}   \end{pmatrix}  \]
and this is zero when $0 \leq k \leq n-1$. If
\begin{equation}   \label{Dn}    D_n = \det  \begin{pmatrix} m_{0,0} & m_{0,1} & m_{0,2} & \cdots & m_{0,n}  \\
                                        m_{1,0} & m_{1,1} & m_{1,2} & \cdots & m_{1,n}  \\
                                        m_{2,0} & m_{2,1} & m_{2,2} & \cdots & m_{2,n}  \\
                                        \vdots & \vdots & \vdots & \cdots & \vdots \\
                                        m_{n,0} & m_{n,1} & m_{n,2} & \cdots & m_{n,n}
                                         \end{pmatrix}  ,
\end{equation}
and $C_n^{-1} = (-2a)^n q^{n(n-1)/2} D_{n-1}$, then $p_n(x)$ defined in \eqref{det} is a monic polynomial. 

\subsection{Al-Salam--Chihara polynomials and little $q$-Laguerre polynomials}   \label{sec:AlSalChi}

The modified moments for Al-Salam--Chihara polynomials can be computed using the integral
\begin{equation}   \label{AlSalChi-int}
    \frac{1}{2\pi} \int_0^\pi \frac{(e^{2i\theta},e^{-2i\theta};q)_\infty}{(ae^{i\theta},ae^{-i\theta},be^{i\theta},be^{-i\theta};q)_\infty}\, d\theta
     = \frac{1}{(q;q)_\infty (ab;q)_\infty}  
\end{equation}
(see \cite[Eq. (15.1.1)]{Ismail}). One easily finds
\begin{eqnarray*}
    m_{k,j} &=& \frac{1}{2\pi} \int_0^\pi (ae^{i\theta},ae^{-i\theta};q)_j (be^{i\theta},be^{-i\theta};q)_k w(\theta;a,b|q)\, d\theta \\
            &=&  \frac{1}{2\pi} \int_0^\pi w(\theta;aq^j,bq^k|q)\, d\theta \\
            &=&  \frac{1}{(q;q)_\infty (abq^{k+j};q)_\infty} = \frac{(ab;q)_{k+j}}{(q;q)_\infty (ab;q)_\infty},
\end{eqnarray*}
so that $m_{k,j} = c_{k+j}$ and $D_n$ given in \eqref{Dn} is a Hankel determinant. The sequence $(c_n)_{n \in \mathbb{N}}$ can be identified as the moments
of a discrete measure. Indeed, by the $q$-binomial theorem \cite[\S 1.3]{GR}
\begin{equation}  \label{q-binom}
   \sum_{k=0}^\infty \frac{(a;q)_k}{(q;q)_k} z^k = \frac{(az;q)_\infty}{(z;q)_\infty}, \qquad |z|<1, 
\end{equation}
we see that (for $a=0$)
\[   \sum_{k=0}^\infty q^{kn} \frac{(ab)^k}{(q;q)_k} = \frac{1}{(abq^n;q)_\infty} = \frac{(ab;q)_n}{(ab;q)_\infty} = c_n (q;q)_\infty, \]
so that 
\[    c_n = \int x^n\, d\mu(x), \]
for the discrete measure $\mu$ on the $q$-lattice $\{q^k; k \in \mathbb{N}\}$ for which
\[   \int_0^1 f(x)\, d\mu(x) = \frac{1}{(q;q)_\infty} \sum_{k=0}^\infty f(q^k) \frac{(ab)^k}{(q;q)_k} .  \]
This is the orthogonality measure for the little $q$-Laguerre polynomials, see \eqref{qLagorth}.
From this we have the following result:

\begin{theorem}  \label{thm:2.1}
Let $T_a: f \mapsto T_a(f)$ be the linear transformation that acts on polynomials as 
\[   T_a(x^k) = (ae^{i\theta},ae^{-i\theta};q)_k.  \]
Then the Al-Salam--Chihara polynomials $p_n(x;a,b|q)$ and little $q$-Laguerre polynomials $q_n(x;a|q)$ are connected by
\[    p_n(x;a,b|q) =  T_a q_n(x;ab/q|q).  \]
\end{theorem}

\begin{proof}
Orthogonal polynomials are given in terms of the moments $c_n$ of their orthogonality measure by the determinant \cite[Eq. (2.1.6)]{Ismail}
\[     p_n(x) =  C_n \det \begin{pmatrix} c_0 & c_1 & c_2 & \cdots & c_n \\
                                       c_1 & c_2 & c_3 & \cdots & c_{n+1} \\
                                       c_2 & c_3 & c_4 & \cdots & c_{n+2} \\
                                       \vdots & \vdots & \vdots & \cdots & \vdots \\
                                       c_{n-1} & c_n & c_{n+1} & \cdots & c_{2n-1} \\
                                        1 & x & x^2 & \cdots & x^n
                   \end{pmatrix} , \]
where $C_n$ is a constant which fixes the normalization. If we compare this with \eqref{det}, then we need to replace every $x^k$ by 
$(ae^{i\theta},ae^{-i\theta};q)_k=A_k(x)$, see \eqref{Ak}. The sequence $(c_n)_{n \in \mathbb{N}}$ contains the moments of the measure
\[ \mu = \frac{1}{(q;q)_\infty} \sum_{k=0}^\infty    \frac{(ab)^k}{(q;q)_k} \delta_{q^k}. \]
Recall that the little $q$-Laguerre polynomials $q_n(x;a|q)$ satisfy the orthogonality relations \eqref{qLagorth},
hence the orthogonal polynomials with moments $(c_n)_{n \in \mathbb{N}}$ are the little $q$-Laguerre polynomials $q_n(x;ab/q|q)$. 
Applying the transformation $T_a$ to the determinantal expression for $q_n(x;ab/q|q)$ then shows that $T_a q_n(x;ab/q|q)$ is proportional
to the Al-Salam--Chihara polynomial $p_n(x;a,b|q)$.
The little $q$-Laguerre polynomials are given by \eqref{qLagsum}
\begin{eqnarray*}
     q_n(x;a|q) &=& {}_2\phi_1\left( \left. \begin{array}{c} q^{-n} , 0 \\ aq \end{array} \right| q,qx \right)  \\
              &=& \sum_{k=0}^n \frac{(q^{-n};q)_k}{(aq;q)_k (q;q)_k} q^k x^k, 
\end{eqnarray*}
hence applying $T_a$ to the polynomial $q_n(x;ab/q|q)$ gives
\begin{eqnarray*}
    T_a q_n(x;ab/q|q) &=&  \sum_{k=0}^\infty \frac{(q^{-n};q)_k}{(ab;q)_k (q;q)_k} q^k (ae^{i\theta};q)_k(ae^{-i\theta};q)_k \\
                       &=&  {}_3\phi_2 \left( \left. \begin{array}{c}  q^{-n}, ae^{i\theta}, ae^{-i\theta} \\ ab, 0  \end{array} \right| q,q \right),
\end{eqnarray*}
and this is indeed the basic hypergeometric expression \eqref{AlSalChi} for the Al-Salam--Chihara polynomial $p_n(x;a,b|q)$.
Therefore the proportionality factor is $1$ and the result follows. 
\end{proof}

The linear transformation $T_a$ can be given explicitly and uses continuous $q$-Hermite polynomials $H_n(x|q)$, which are given in \cite[\S 14.26]{KLS}
\cite[\S 13.1]{Ismail}.
They satisfy the orthogonality
\begin{equation}   \label{qHerorth}
   \frac{1}{2\pi} \int_{-1}^1 H_n(x|q) H_m(x|q) (e^{2i\theta},e^{-2i\theta};q)_\infty \, \frac{dx}{\sqrt{1-x^2}} 
 = \frac{\delta_{m,n}}{(q^{n+1};q)_\infty},  
\end{equation}
and they have the generating function \cite[Eq. (14.26.11)]{KLS} \cite[Thm. 13.1.1]{Ismail} 
\begin{equation}  \label{qHergen}
    \sum_{n=0}^\infty \frac{H_n(x|q)}{(q;q)_n} t^n = \frac{1}{(te^{i\theta},te^{-i\theta};q)_\infty}, \qquad x= \cos \theta. 
\end{equation}

\begin{theorem}   \label{thm:2.2}
The linear transformation $T_a$ that acts on polynomials as
\[   T_a x^k = (ae^{i\theta},ae^{-i\theta};q)_k \]
is given by
\begin{equation}    \label{Ta}
    (T_a f)(x) = (ae^{i\theta},ae^{-i\theta};q)_\infty \sum_{n=0}^\infty f(q^n) a^n \frac{H_n(x|q)}{(q;q)_n} .  
\end{equation}
\end{theorem}

\begin{proof}
Obviously the transformation given in \eqref{Ta} is linear, so we only need to check that it acts properly on the monomials $x^k$.
If we take $f(x) = x^k$ in \eqref{Ta} then
\[    T_a x^k = (ae^{i\theta},ae^{-i\theta};q)_\infty \sum_{n=0}^\infty q^{kn} a^n \frac{H_n(x|q)}{(q;q)_n}.  \]
Taking $t=aq^k$ in the generating function \eqref{qHergen} gives
\[   T_a x^k = \frac{(ae^{i\theta},ae^{-i\theta};q)_\infty}{(aq^ke^{i\theta},aq^ke^{-i\theta};q)_\infty} = (ae^{i\theta},ae^{-i\theta};q)_k, \]
which is indeed what we need.
\end{proof}

The transformation $T_a$ has an interesting isometric property, preserving certain inner products.
\begin{proposition}  \label{prop:3.3}
Let $\langle f,g \rangle_{\textup{dis}(a)}$ be the discrete inner product
\[    \langle f,g \rangle_{\textup{dis}(a)} = \frac{1}{(q;q)_\infty} \sum_{n=0}^\infty f(q^n) g(q^n) \frac{a^n}{(q;q)_n} \ , \]
and $\langle u,v \rangle_{\textup{cont}(a,b)}$ be the continuous inner product
\[   \langle u,v \rangle_{\textup{cont}(a,b)} = \frac{1}{2\pi} \int_{-1}^1 u(x)v(x) w(\theta;a,b|q) \frac{dx}{\sqrt{1-x^2}} \ , \]
where $w(\theta;a,b|q)$ is the weight function \eqref{AlSalChiw}. Then
\[    \langle T_af , T_bg \rangle_{\textup{cont}(a,b)} = \langle f,g \rangle_{\textup{dis}(ab)}.  \]
\end{proposition}

\begin{proof}
By using \eqref{Ta} we find
\begin{multline*}
   \langle T_af , T_bg \rangle_{\textup{cont}(a,b)} = \sum_{n=0}^\infty \sum_{m=0}^\infty f(q^n)g(q^m) \frac{a^n b^m}{(q;q)_n (q;q)_m}  \\
  \times  \frac{1}{2\pi} \int_{-1}^1 (ae^{i\theta},ae^{-i\theta},be^{i\theta},be^{-i\theta};q)_\infty w(\theta;a,b|q) H_n(x|q) H_m(x|q) \frac{dx}{\sqrt{1-x^2}}. 
\end{multline*}
The integral simplifies to
\[     \frac{1}{2\pi} \int_{-1}^1 (e^{2i\theta},e^{-2i\theta};q)_\infty  H_n(x|q) H_m(x|q) \frac{dx}{\sqrt{1-x^2}} =\delta_{n,m} \frac{(q;q)_n}{(q;q)_\infty} \ ,  \]
which follows from the orthogonality \eqref{qHerorth} of the continuous $q$-Hermite polynomials. Hence the double sum becomes a single sum and the result follows.
\end{proof}

As a corollary we see that the orthogonality relations of the Al-Salam--Chihara polynomials follow from the orthogonality of the
little $q$-Laguerre polynomials. Indeed, by Theorem \ref{thm:2.1} we have $T_a q_n(x;ab/q|q) = p_n(x;a,b|q)$ and by interchanging $a$ and $b$ 
we also have $T_b q_m(x;ab/q|q) = p_m(x;b,a|q)$ and the latter is equal to $p_m(x;a,b|q)$. Proposition \ref{prop:3.3} then shows that
\begin{multline*}
  \frac{1}{2\pi} \int_{-1}^1 p_n(x;a,b|q)p_m(x;a,b|q) w(\theta;a,b|q)\, d\theta \\
   = \frac{1}{(q;q)_\infty}  \sum_{k=0}^\infty     q_n(q^k;ab/q|q) q_m(q^k;ab/q|q) \frac{(ab)^k}{(q;q)_k}, 
\end{multline*}
and the latter is $0$ whenever $m \neq n$ by the orthogonality \eqref{qLagorth} of the little $q$-Laguerre polynomials.
Clearly the norms of $p_n(x;a,b|q)$ and $q_n(x;ab/q|q)$ are also connected.

\subsection{Askey--Wilson polynomials and little $q$-Jacobi polynomials}   \label{sec:AW}
The modified moments $m_{k,j}$ for the Askey--Wilson weight are given by
\begin{eqnarray*}
     m_{k,j} &=& \frac{1}{2\pi} \int_0^\pi (ae^{i\theta},ae^{-i\theta};q)_j (be^{i\theta},be^{-i\theta};q)_k w(\theta;a,b,c,d|q)\, d\theta \\
            &=&  \frac{1}{2\pi} \int_0^\pi w(\theta;aq^j,bq^k,c,d|q)\, d\theta.  
\end{eqnarray*}
To evaluate this integral we use the Askey--Wilson integral \cite[Eq. (15.2.1)]{Ismail}
\begin{equation}  \label{AW-int}
  \frac{1}{2\pi} \int_0^\pi w(\theta;a,b,c,d|q)\, d\theta = \frac{(abcd;q)_\infty}{(q;q)_\infty (ab,ac,ad,bc,bd,cd;q)_\infty}, \qquad |a|,|b|,|c|,|d|<1,
\end{equation}
which gives
\begin{eqnarray}  \label{modmomAW}
   m_{k,j} &=& \frac{(abcdq^{k+j};q)_\infty}{(q;q)_\infty (abq^{k+j},acq^j,adq^j,bcq^k,bdq^k,cd;q)_\infty}   \nonumber \\ 
           &=& \frac{(abcd;q)_\infty}{(q;q)_\infty (ab,ac,ad,bc,bd,cd;q)_\infty}  \frac{(ab;q)_{k+j} (ac;q)_j (ad;q)_j(bc;q)_k(bd;q)_k}{(abcd;q)_{k+j}}.
  \quad
\end{eqnarray}
This expression contains a Hankel part $c_{k+j}$ depending only on $k+j$, with
\[     c_n = c_0 \frac{(ab;q)_n}{(abcd;q)_n}, \qquad   c_0 =    \frac{(abcd;q)_\infty}{(q;q)_\infty (ab,ac,ad,bc,bd,cd;q)_\infty} .  \]
The sequence $(c_n)_{n \in \mathbb{N}}$ can again be identified as the moments of a discrete measure on the $q$-lattice $\{ q^k : k \in \mathbb{N}\}$.
If we take $z=bq^n$ in the $q$-binomial theorem \eqref{q-binom}, then
\[   \sum_{k=0}^\infty \frac{(a;q)_k}{(q;q)_k} b^k q^{nk} = \frac{(abq^n;q)_\infty}{(bq^n;q)_\infty} = \frac{(ab;q)_\infty}{(b;q)_\infty} \frac{(b;q)_n}{(ab;q)_n}, \]
so that 
\[   c_n = \int_0^1 x^n \, d\nu(x), \qquad  \int_0^1 f(x)\, d\nu(x) = \hat{c}_0 \sum_{k=0}^\infty \frac{(cd;q)_k}{(q;q)_k} (ab)^k f(q^k), \]
where $\hat{c}_0 = 1/(q;ac,ad,bc,bd,cd;q)_\infty$. This is the orthogonality measure for little $q$-Jacobi polynomials,
see \eqref{qJacorth}. We can now prove the following result.

\begin{theorem}   \label{thm:2.4}
Let $T_{a,c,d}: f \mapsto T_{a,c,d}(f)$ be the linear transformation that acts on polynomials as 
\[   T_{a,c,d}(x^k) = \frac{(ae^{i\theta},ae^{-i\theta};q)_k}{(ac;q)_k(ad;q)_k}.  \]
Then the Askey--Wilson polynomials $p_n(x;a,b,c,d|q)$ and little $q$-Jacobi polynomials $q_n(x;a,b|q)$ are connected by
\[    p_n(x;a,b,c,d|q) = T_{a,c,d} q_n(x;ab/q,cd/q|q).  \]
\end{theorem}

\begin{proof}
If we insert \eqref{modmomAW} in the determinant \eqref{det}, then we can take out the factor $(bc;q)_k(bd;q)_k$ in the $k$th row and the factor
$(ac;q)_j(ad;q)_j$ in the $j$th column, to find
\[  p_n(x;a,b,c,d|q) =  C_n \prod_{k=0}^{n-1} (bc,bd;q)_k \prod_{j=0}^n (ac,ad;q)_j 
   \det \begin{pmatrix}
       c_0 & c_1 & c_2 & \cdots & c_n \\
        c_1 & c_2 & c_3 & \cdots & c_{n+1} \\
        c_2 & c_3 & c_4 & \cdots & c_{n+2} \\
          \vdots & \vdots & \vdots & \cdots & \vdots \\
        \hat{A}_0(x) & \hat{A}_1(x) & \hat{A}_2(x) & \cdots & \hat{A}_n(x) \end{pmatrix} ,  \]
where
\[    \hat{A}_k(x) = \frac{(ae^{i\theta},ae^{-i\theta};q)_k}{(ac;q)_k(ad;q)_k}.  \]
Recall that the little $q$-Jacobi polynomials satisfy the orthogonality \eqref{qJacorth}
\[   \sum_{k=0}^\infty q_n(q^k;a,b|q) q_m(q^k;a,b|q) \frac{(bq;q)_k}{(q;q)_k} (aq)^k = 0, \qquad m \neq n, \]
so the sequence $(c_n)_{n \in \mathbb{N}}$ contains the moments of the orthogonality measure for the little $q$-Jacobi polynomials
$q_n(x;ab/q,cd/q|q)$. The determinant representation of the little $q$-Jacobi polynomials is therefore given by
\[    q_n(x;ab/q,cd/q|q) = \widehat{C}_n \det \begin{pmatrix}
       c_0 & c_1 & c_2 & \cdots & c_n \\
        c_1 & c_2 & c_3 & \cdots & c_{n+1} \\
        c_2 & c_3 & c_4 & \cdots & c_{n+2} \\
          \vdots & \vdots & \vdots & \cdots & \vdots \\
         1 & x & x^2 & \cdots & x^n \end{pmatrix}, \]
where $\widehat{C}_n$ is a normalizing constant. Applying the linear transformation $T_{a,c,d}$ shows that $T_{acd} q_n(x;ab/q,cd/q|q)$ is
proportional to $p_n(x;a,b,c,d|q)$.
The explicit expression for the little $q$-Jacobi polynomials is \eqref{qJacsum}
hence applying the transformation $T_{a,c,d}$ to it gives
\begin{eqnarray*}  T_{a,c,d} q_n(x;ab/q,cd/q|q) &=&   \sum_{k=0}^n \frac{(q^{-n},q)_k (abcdq^{n-1};q)_k}{(ab;q)_k(q;q)_k} q^k 
         \frac{(ae^{i\theta},ae^{-i\theta};q)_k}{(ac;q)_k(ad;q)_k} \\
     &=&  {}_4\phi_3 \left( \left. \begin{array}{c} q^{-n} , abcdq^{n-1} , ae^{i\theta}, ae^{-i\theta} \\
                                                      ab , ac, ad \end{array} \right| q,q \right), 
\end{eqnarray*}
which is indeed the basic hypergeometric expression \eqref{AW} for the Askey--Wilson polynomial. Hence the proportionality factor is 1 and the result follows.
\end{proof}

The linear transformation $T_{a,c,d}$ can also be given explicitly and is in terms of the Al-Salam--Chihara polynomials. We will use a multiple of the
polynomials $p_n(x;a,b|q)$ defined above and put
\[   Q_n(x;a,b|q) = \frac{(ab;q)_n}{a^n} p_n(x;a,b|q) 
   = \frac{(ab;q)_n}{a^n} {}_3\phi_2 \left( \left.\begin{array}{c} q^{-n}, ae^{i\theta}, ae^{-i\theta} \\ ab, 0 \end{array} \right| q, q \right) .  \]
They have the following generating function (\cite[Eq. (14.8.13)]{KLS} \cite[Eq. (15.1.10)]{Ismail}
\begin{equation}   \label{AlSalChigen}
    \sum_{n=0}^\infty \frac{Q_n(x;a,b|q)}{(q;q)_n} t^n = \frac{(at,bt;q)_\infty}{(te^{i\theta},te^{-i\theta};q)_\infty}, \qquad x=\cos \theta.  
\end{equation}

\begin{theorem}   \label{thm:2.5}
The linear transformation $T_{a,c,d}$ that acts on polynomials as
\[   T_{a,c,d} x^k = \frac{(ae^{i\theta},ae^{-i\theta};q)_k}{(ac;q)_k(ad;q)_k} \]
is given by
\begin{equation}  \label{Tacd}
    (T_{a,c,d} f )(x) = \frac{(ae^{i\theta},ae^{-i\theta};q)_\infty}{(ac,ad;q)_\infty} \sum_{n=0}^\infty f(q^n) a^n \frac{Q_n(x;c,d|q)}{(q;q)_n}. 
\end{equation}
\end{theorem}

\begin{proof}
The transformation in \eqref{Tacd} is obviously linear, so we need only to check how it acts on polynomials $x^k$. Taking $f(x)=x^k$ in \eqref{Tacd}
gives
\[   T_{a,c,d} x^k =  \frac{(ae^{i\theta},ae^{-i\theta};q)_\infty}{(ac,ad;q)_\infty} \sum_{n=0}^\infty q^{kn} a^n \frac{Q_n(x;c,d|q)}{(q;q)_n} , \]
and if we put $t=aq^k$ in the generating function \eqref{AlSalChigen}, then this gives
\[  T_{a,c,d} x^k =  \frac{(ae^{i\theta},ae^{-i\theta};q)_\infty}{(ac,ad;q)_\infty} 
                     \frac{(acq^k,adq^k;q)_\infty}{(aq^ke^{i\theta},aq^ke^{-i\theta};q)_\infty} = \frac{(ae^{i\theta},ae^{-i\theta};q)_k}{(ac,ad;q)_k}, \]
which is the desired result.
\end{proof}
    
This transformation was given explicitly in \cite{IsmailKoel}, see e.g. their equation (4.5).  Their formula (4.7) also gives the Askey--Wilson
polynomials as the image of applying $T_{a,c,d}$ to little $q$-Jacobi polynomials.
When $c=d=0$ the transformation $T_{a,c,d}$ is equal to $T_a$, which reflects the fact that $Q_n(x;0,0|q) = H_n(x|q)$. 
This transformation $T_{a,c,d}$ also obeys a Plancherel type result in the following sense.
\begin{proposition}  \label{prop:2.6}
Let $\langle f,g \rangle_{\textup{dis}(a,b)}$ be the discrete inner product
\[    \langle f,g \rangle_{\textup{dis}(a,b)} = \frac{1}{(q;q)_\infty} \sum_{n=0}^\infty f(q^n) g(q^n) a^n \frac{(b;q)_n}{(q;q)_n}, \]
and $\langle u,v \rangle_{\textup{cont}(a,b,c,d)}$ be the continuous inner product
\[   \langle u,v \rangle_{\textup{cont}(a,b,c,d)} = \frac{(ac,ad,bc,bd,cd;q)_\infty}{2\pi} \int_{-1}^1 u(x)v(x) w(\theta;a,b,c,d|q) \frac{dx}{\sqrt{1-x^2}}, \]
where $w(\theta;a,b,c,d|q)$ is the Askey--Wilson weight function \eqref{AWw}. Then
\[    \langle T_{a,c,d}f , T_{b,c,d}g \rangle_{\textup{cont}(a,b,c,d)} = \langle f,g \rangle_{\textup{dis}(ab,cd)} .  \]
\end{proposition}

\begin{proof}
If we use the expression \eqref{Tacd} for the transformation $T_{a,c,d}$ then
\begin{multline*}  \langle T_{a,c,d} f, T_{b,c,d} g \rangle_{a,b,c,d} =
   \sum_{n=0}^\infty \sum_{m=0}^\infty f(q^n)g(q^m) \frac{a^n b^m}{(q;q)_n (q;q)_m (ac,ad,bc,bd;q)_\infty} \\
  \times \frac{1}{2\pi} \int_0^\pi (ae^{i\theta},ae^{-i\theta};q)_\infty (be^{i\theta},be^{-i\theta};q)_\infty
    Q_n(x;c,d|q)Q_m(x;c;d|q) w(\theta;a,b,c,d)\, d\theta.  
\end{multline*}
Observe that the integral in this expression is
\begin{multline*}   \frac{1}{2\pi} \int_0^\pi Q_n(x;c,d|q) Q_m(x;c,d|q) 
   \frac{(e^{2i\theta},e^{-2i\theta};q)_\infty}{(ce^{i\theta},ce^{-i\theta},de^{i\theta},de^{-i\theta};q)_\infty} \,d\theta \\
   =  \frac{\delta_{n,m}}{(q^{n+1};q)_\infty (cdq^n;q)_\infty}, 
\end{multline*}
(see \cite[Eq. (14.8.2)]{KLS}), so that the double sum becomes a single sum
\[   \frac{1}{(q;q)_\infty (ac,ad,bc,bd,cd;q)_\infty} \sum_{n=0}^\infty f(q^n)g(q^n) (ab)^n \frac{(cd;q)_n}{(q;q)_\infty}, \]
which is the desired expression in terms of the discrete inner product $\langle f,g \rangle_{\textup{dis}(ab,cd)}$. 
\end{proof}

As a corollary, the orthogonality for the Askey--Wilson polynomials now follows from the orthogonality of the little $q$-Jacobi polynomials.
Indeed, if we use Theorem \ref{thm:2.4} then $T_{a,c,d} q_n(x;ab/q,cd/q|q) = p_n(x;a,b,c,d|q)$ and $T_{b,c,d} q_m(x;ab/q,cd/q|q) = p_m(x;b,a,c,d|q)$ and the latter is equal to $p_m(x;a,b,c,d|q)$. Hence by Proposition \ref{prop:4.3}
\begin{multline*}  \frac{(ac,ad,bc,bd,cd;q)_\infty}{2\pi} \int_0^\pi  p_n(x;a,b,c,d|q) p_m(x;a,b,c,d|q) w(\theta;a,b,c,d|q)\ d\theta \\ 
 =   \frac{1}{(q;q)_\infty} \sum_{k=0}^\infty q_n(q^k;ab/q,cd/q|q) q_m(q^k;ab/q,cd/q|q) (ab)^k \frac{(cd;q)_k}{(q;q)_k} ,
\end{multline*}
and this is $0$ whenever $m\neq n$ because of the orthogonality relations \eqref{qJacorth} for the little $q$-Jacobi polynomials.
Clearly one can also relate the norms of $p_n(x;a,b,c,d|q)$ with those of $q_n(x;ab/q,cd/q|q)$ by putting $^n=m$.

\subsection{Continuous dual $q$-Hahn polynomials and little $q$-Laguerre polynomials}  \label{sec:qHahn}
The modified moments for the continuous dual $q$-Hahn polynomials correspond to the modified moments of the Askey--Wilson weight with $d=0$,
\begin{eqnarray}  \label{momqHahn}
       m_{k,j} &=& \frac{1}{2\pi} \int_0^\pi (ae^{i\theta},ae^{-i\theta};q)_j ((be^{i\theta},be^{-i\theta};q)_k w(\theta|a,b,c|q)\, d\theta \nonumber \\
              &=&  \frac{(ab;q)_{k+j}(ac;q)_j (bc;q)_k}{(q;q)_\infty (ab,ac,bc;q)_\infty} .  
\end{eqnarray}
The Hankel part $(ab;q)_{k+j}$ corresponds to the moments of the discrete measure $\mu$ that we used in Section \ref{sec:AlSalChi},
which is the orthogonality measure for the little $q$-Laguerre polynomials. We can then prove the following result

\begin{theorem}    \label{thm:2.7}
Let $T_{a,c}: f \mapsto T_{a,c} f$ be the linear transformation that acts on polynomials as
\[   T_{a,c} x^k = \frac{(ae^{i\theta},ae^{-i\theta};q)_k}{(ac;q)_k}.  \]
Then the continuous dual $q$-Hahn polynomials $p_n(x;a,b,c|q)$ and the little $q$-Laguerre polynomials $q_n(x;a|q)$ are connected by
\[    p_n(x;a,b,c|q) = T_{a,c} q_n(x;ab/q|q) .  \]
\end{theorem}

\begin{proof}
If we insert the modified moments \eqref{momqHahn} in \eqref{det}, then we can take out the factor $(bc;q)_k$ in row $k$ and the factor
$(ac;q)_j$ in column $j$. This gives
\[    p_n(x;a,b,c|q) = C_n \prod_{k=0}^{n-1} (bc;q)_k \prod_{j=0}^n  (ac;q)_j \det \begin{pmatrix}
                     c_0 & c_1 & c_2 & \cdots & c_n \\
                     c_1 & c_2 & c_3 & \cdots & c_{n+1} \\
                     \vdots & \vdots & \vdots & \cdots & \vdots \\
                     \tilde{A}_0(x) & \tilde{A}_1(x) & \tilde{A}_2(x) & \cdots & \tilde{A}_n(x) \end{pmatrix},  \]
where
\[   \tilde{A}_k(x) = \frac{(ae^{i\theta},ae^{-i\theta};q)_k}{(ac;q)_k},  \]
and $(c_n)_{n\in \mathbb{N}}$ are the moments of the measure $\mu$, which is the measure for the little $q$-Laguerre polynomials
$q_n(x;ab/q|q)$. Hence applying $T_{a,c}$ to $q_n(x;ab/q|q)$ gives the continuous dual $q$-Hahn polynomials.
\end{proof}

The linear transformation $T_{a,c}$ can again be given explicitly and is in term of the continuous big $q$-Hermite polynomials
$H_n(x;a|q)$, which are given in \cite[\S 14.18]{KLS}. The orthogonality relations are
\[   \frac{1}{2\pi} \int_{-1}^1 H_n(x;a|q)H_m(x;a|q)  \frac{(e^{2i\theta},e^{-2i\theta};q)_\infty}{(ae^{i\theta},ae^{-i\theta};q)_\infty} 
 \frac{dx}{\sqrt{1-x^2}} = \frac{\delta_{m,n}}{(q^{n+1},q)_\infty} , \]
and they have a generating function
\begin{equation}   \label{bigqHergen}
  \sum_{n=0}^\infty  \frac{H_n(x;a|q)}{(q;q)_n} t^n = \frac{(at;q)_\infty}{(te^{i\theta},te^{-i\theta};q)_\infty}, \qquad x=\cos \theta.  
\end{equation}Observe that for $a=0$ they reduce to the continuous $q$-Hermite polynomials that we used in Section \ref{sec:AlSalChi}.

\begin{theorem}   \label{thm:2.8}
The linear transformation $T_{a,c}$ that acts on polynomials as
\[   T_{a,c} x^k = \frac{(ae^{i\theta},ae^{-i\theta};q)_k}{(ac;q)_k}  \]
is given by
\begin{equation}   \label{Tac}
   (T_{a,c} f)(x) = \frac{(ae^{i\theta},ae^{-i\theta};q)_\infty}{(ac;q)_\infty} \sum_{n=0}^\infty f(q^n) a^n \frac{H_n(x;c|q)}{(q;q)_n} .  
\end{equation}
\end{theorem}

\begin{proof}
The linearity is obvious and the action of $T_{a,c}$ on $x^k$ can easily be checked, using the generating function \eqref{bigqHergen}.
\end{proof}

\begin{proposition} \label{prop:2.9}
Let $\langle f,g \rangle_{\textup{dis}(a)}$ be the discrete inner product
\[   \langle f,g \rangle_{\textup{dis}(a)} = \frac{1}{(q;q)_\infty} \sum_{n=0}^\infty f(q^n)g(q^n) \frac{a^n}{(q;q)_n}, \]
and $\langle u,v \rangle_{\textup{cont}(a,b,c)}$ the continuous inner product
\[   \frac{(ac,bc;q)_\infty}{2\pi} \int_{-1}^1 u(x)v(x) w(\theta;a,b,c|q)\, d\theta, \]
where $w(\theta;a,b,c|q)$ is the weight function \eqref{qHahnw}. Then
\[   \langle T_{a,c} f ,T_{b,c} g \rangle_{\textup{cont}(a,b,c)} = \langle f,g \rangle_{\textup{dis}(ab)} . \]
\end{proposition}

\begin{proof}
This follows from Proposition \ref{prop:2.6} by taking $d=0$, or from the orthogonality of the continuous big $q$-Hermite polynomials in a similar
way as in the proofs of Proposition \ref{prop:3.3} and \ref{prop:2.6}.
\end{proof}

As a corollary one can deduce the orthogonality of the continuous dual $q$-Hahn polynomials from the orthogonality of the little $q$-Laguerre
polynomials, by using Theorem \ref{thm:2.7}.

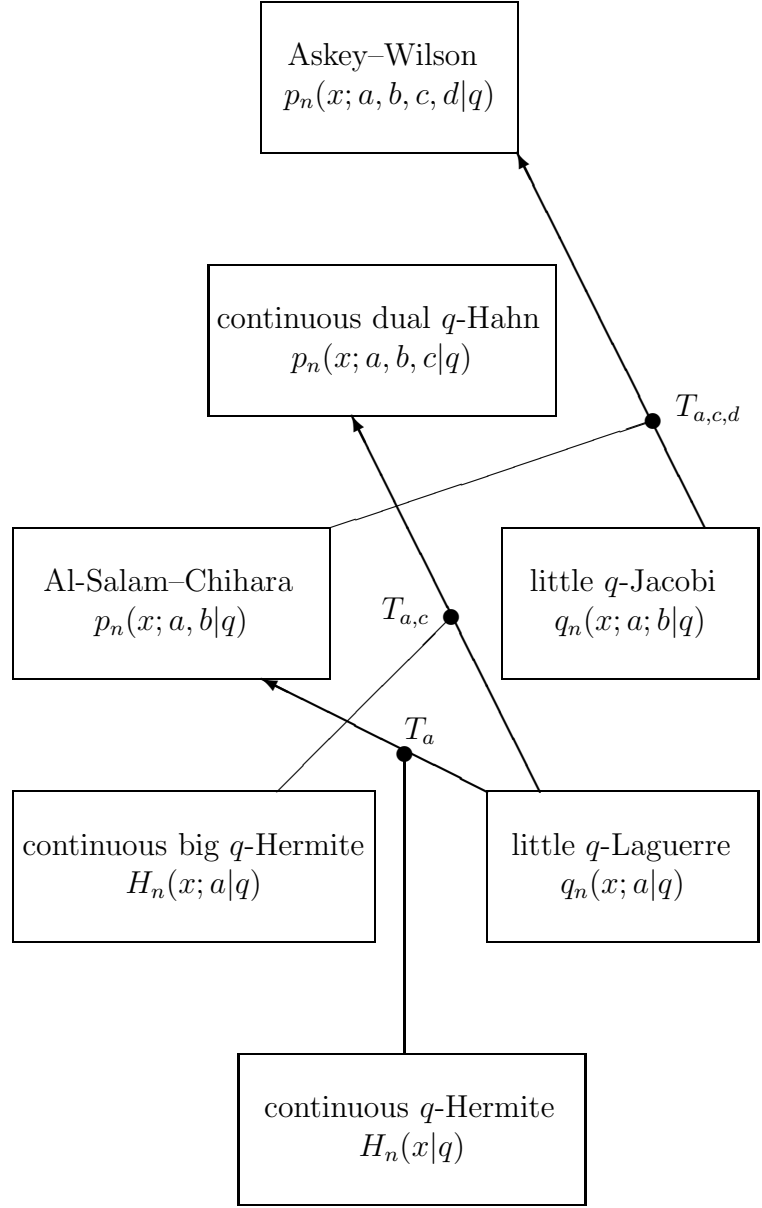
\begin{figure}[ht]
\centering
\unitlength=1mm
\begin{picture}(100,160)(0,20)
\put(33,160){\framebox(34,20){\parbox{1.2in}{Askey--Wilson \\ \centering $p_n(x;a,b,c,d|q)$}}}
\put(26,125){\framebox(46,20){\parbox{1.8in}{continuous dual $q$-Hahn \\ \centering $p_n(x;a,b,c|q)$}}}                    
\put(0,90){\framebox(42,20){\parbox{1.4in}{Al-Salam--Chihara \\ \centering $p_n(x;a,b|q)$}}}
\put(0,55){\framebox(48,20){\parbox{1.8in}{continuous big $q$-Hermite \\ \centering $H_n(x;a|q)$}}}
\put(65,90){\framebox(34,20){\parbox{1.2in}{little $q$-Jacobi \\ \centering $q_n(x;a;b|q)$}}}
\put(63,55){\framebox(36,20){\parbox{1.2in}{little $q$-Laguerre \\ \centering $q_n(x;a|q)$}}}
\put(30,20){\framebox(46,20){\parbox{1.6in}{continuous $q$-Hermite \\ \centering $H_n(x|q)$}}}
\put(88,125){$T_{a,c,d}$}
\put(52,82){$T_a$}
\put(42,110){\line(3,1){43}}
\put(52,80){\circle*{2}}
\put(85,124.3){\circle*{2}}
\put(52,40){\line(0,1){40}}
\put(35,75){\line(1,1){23}}
\put(58.2,98.2){\circle*{2}}
\put(49,98){$T_{a,c}$}
\thicklines
\put(63,75){\vector(-2,1){30}}
\put(70,75){\vector(-1,2){25}}
\put(92,110){\vector(-1,2){25}}
\end{picture}
\caption{Schematic presentation of the hypergeometric orthogonal polynomials and the transformations between them.}
\label{fig:1}
\end{figure}

It is interesting to see how these various families of basic hypergeometric polynomials are connected, see Figure \ref{fig:1}.
Continuous $q$-Hermite polynomials are at the bottom of the $q$-Askey scheme of basic hypergeometric polynomials \cite[p.~413]{KLS},
with a weight function $w(\theta|q) = (e^{2i\theta},e^{-2i\theta};q)_\infty$. You need them for the transformation $T_a$ that maps
little $q$-Laguerre polynomials to Al-Salam--Chihara polynomials, which have a weight $w(\theta;a,b|q)$ with two parameters $a,b$.
A more general transformation $T_{a,c}$ maps the same little $q$-Laguerre polynomials to continuous dual $q$-Hahn polynomials, which have
a weight function $w(\theta;a,b,c|q)$ with three parameters.
The Al-Salam--Chihara polynomials in turn are needed for the transformation $T_{a,c,d}$ which maps little $q$-Jacobi polynomials
to Askey--Wilson polynomials which have a weight function $w(\theta;a,b,c,d|q)$ with four parameters.
These transformations $T_a$, $T_{a,c}$ and $T_{a,c,d}$ seem to be special cases of the Askey--Wilson function transform given in
\cite[Eq. (5.9)]{KoelStok} and \cite[p.~312]{Stok}, since the polynomials $H_n(x|q)$, $H_n(x;a|q)$ and $Q_n(x;a,b|q)$ in our
tranformations $T_a, T_{a,c}$ and $T_{a,c,d}$ are special cases of the Askey--Wilson polynomial $p_n(x;a,b,c,d|q)$, but this needs a little more inspection of those papers. Our Plancherel type formulas then correspond to \cite[Thm.~1 and Prop.~3]{KoelStok} and \cite[Thm. 5.1]{Stok}.

\section{Multiple basic hypergeometric polynomials}   \label{sec:3}
We will now use the results from the previous section to construct multiple hypergeometric orthogonal polynomials for Askey--Wilson
weights, continuous dual $q$-Hahn weights and Al-Salam--Chihara weights from discrete multiple orthogonal polynomials on the $q$-lattice
$\{q^k, k \in \mathbb{N}\}$. We always use $r$ weights $(w_1,\ldots,w_r)$ obtained by changing the parameter $a$ to a vector $\vec{a}=(a_1,\ldots,a_r)$.
We start by recalling some results for the multiple little $q$-Laguerre and the multiple little $q$-Jacobi polynomials.

\subsection{Little $q$-Laguerre polynomials}  \label{sec:3.1}
The little $q$-Laguerre polynomials are given by \cite[\S 14.20]{KLS}
\begin{equation}   \label{qLag}
   q_n(x;a|q) = {}_2\phi_1 \left( \left. \begin{array}{c} q^{-n} , 0 \\ aq \end{array} \right| q, qx \right) =
   \sum_{k=0}^n \frac{(q^{-n};q)_k}{(q;q)_k (aq;q)_k} q^kx^k .
\end{equation}
When taking $a = q^\alpha$, they can be obtained by the Rodrigues formula
\begin{equation}   \label{qLagRod}
    (qx;q)_\infty x^\alpha q_n(x;a|q) = \frac{q^{\binom{n}{2}}a^n(1-q)^n}{(aq;q)_n} D_p^{n} x^{n+\alpha} (qx;q)_\infty , 
\end{equation}
where $p=1/q$ and $D_p$ is the $p$-difference operator
\[    D_p f(x) = \frac{f(px)-f(x)}{x(p-1)} = \frac{q}{1-q}  \frac{f(px)-f(x)}{x}. \]
Note that these polynomials are neither monic nor orthonormal, but they are normalized by $q_n(0;a|q) = 1$.

Multiple little $q$-Laguerre polynomials $q_{\vec{n}}(x;\vec{a}|q)$ for the multi-index
$\vec{n}=(n_1,n_2,\ldots,n_r)$ depend on $r$ parameters $\vec{a} = (a_1,a_2,\ldots,a_r)$ and
satisfy multiple orthogonality relations
\begin{equation}  \label{mqLagorth}
   \sum_{k=0}^\infty q_{\vec{n}}(q^k;\vec{a}|q) q^{k\ell}  \frac{(a_jq)^k}{(q;q)_k} = 0, \qquad 0 \leq k \leq n_j-1, 
\end{equation}
for $1 \leq j \leq r$. They can be defined by the Rodrigues formula
\begin{equation}   \label{mqLagRod}
     (qx;q)_\infty q_{\vec{n}}(x;\vec{a}|q) = C_{\vec{n}}(\vec{a}) \prod_{j=1}^r \left( x^{-\alpha_j} D_p^{n_j}
       x^{n_j+\alpha_j} \right)  \ (qx;q)_\infty, 
\end{equation}
where $a_j = q^{\alpha_j}$ and the $p$-difference operators $x^{-\alpha_j} D_p^{n_j}  x^{n_j+\alpha_j}$ are commuting. 

\begin{theorem}   \label{thm:3.1}
If we take the normalizing factor in \eqref{mqLagRod} as
\begin{equation}  \label{Cn}
  C_{\vec{n}}(\vec{a}) = (1-q)^{|\vec{n}|} \prod_{j=1}^r \frac{q^{\binom{n_j}{2}} a_j^{n_j}}{(a_jq;q)_{n_j}},
\end{equation}
then an explicit expression for the multiple little $q$-Laguerre polynomials is given by
\begin{equation} \label{mqLag}
   q_{\vec{n}}(x;\vec{a}|q) 
= \sum_{k_1=0}^{n_1} \cdots \sum_{k_r=0}^{n_r}  \prod_{j=1}^r \frac{(q^{-n_j};q)_{k_j}}{(q;q)_{k_j}(a_jq;q)_{k_j}}  \prod_{j=1}^{r-1} \frac{(a_jq^{n_j+1};q)_{k_{j+1}+\cdots+k_r}}{(a_jq^{k_j+1};q)_{k_{j+1}+\cdots+k_r}}
       \frac{(qx)^{|\vec{k}|}}{q^{\sum_{j=1}^r n_j \sum_{i=j+1}^r k_i}} .
\end{equation}
These multiple little $q$-Laguerre polynomials are normalized so that $q_{\vec{n}}(0;\vec{a}|q)=1$.
\end{theorem}

\begin{proof}
We will use induction on $r$. For $r=1$ we have the result \eqref{qLag} for the usual little $q$-Laguerre polynomials in
\cite[\S 14.20]{KLS} and
\begin{equation}   \label{Cr1}
  C_n(a) = \frac{q^{\binom{n}{2}}a^n(1-q)^n}{(aq;q)_n},
\end{equation}
as can be deduced from \eqref{qLagRod}.

Suppose the result is true for $r-1$. The difference operators $x^{-\alpha_j} D_p^{n_j} x^{n_j+\alpha_j}$,
$1 \leq j \leq r$ are all commuting, so the order in which we take the product of these operators
is irrelevant. The Rodrigues formula \eqref{mqLagRod} can then be written as
\[   (qx;q)_\infty q_{\vec{n}}(x;\vec{a}|q) = C_{\vec{n}}(\vec{a}) x^{-\alpha_1} D_p^{n_1} x^{n_1+\alpha_1}
    \prod_{j=2}^r \left( x^{-\alpha_j} D_p^{n_j} x^{n_j+\alpha_j} \right) (qx;q)_\infty. \]
By the induction hypothesis, the product $\prod_{j=2}^r$ can be written as an $(r-1)$-fold sum and we have
\begin{eqnarray*}
  (qx;q)_\infty q_{\vec{n}}(x;\vec{a}|q) &=& \frac{C_{\vec{n}}(\vec{a})}{C_{\vec{n}-n_1\vec{e}_1}(\vec{a}-a_1\vec{e}_1)}
   x^{-\alpha_1} D_p^{n_1} x^{n_1+\alpha_1} (qx;q)_\infty q_{\vec{n}-n_1\vec{e}_1}(x;\vec{a}-a_1\vec{e}_1|q) \\
   & = & \frac{C_{\vec{n}}(\vec{a})}{C_{\vec{n}-n_1\vec{e}_1}(\vec{a}-a_1\vec{e}_1)}
         \sum_{k_2=0}^{n_2} \cdots \sum_{k_r=0}^{n_r} \prod_{j=2}^r \frac{(q^{-n_j};q)_{k_j}}{(q;q)_{k_j} (a_jq;q)_{k_j}} \\
     & &  \times   \prod_{j=2}^{r-2} \frac{a_jq^{n_j+1};q)_{k_{j+1}+\cdots+k_r}}{(a_jq^{k_j+1};q)_{k_{j+1}+\cdots+k_r}} 
     \frac{(qx)^{k_2+\cdots+k_r}}{q^{\sum_{j=2}^r n_j \sum_{i=j+1}^r k_i}}  \\
     & & \times \   x^{-\alpha_1} D_p^{n_1} x^{n_1+\alpha_1+k_2+\cdots+k_r} (qx;q)_\infty. 
\end{eqnarray*}
 The $p$-difference on the last line can be worked out using the Rodrigues formula \eqref{qLagRod} for $r=1$
\[   x^{-\alpha_1} D_p^{n_1} x^{n_1+\alpha_1+k_2+\cdots+k_r} (qx;q)_\infty = \frac{(qx;q)_\infty}{C_{n_1}(a_1q^{k_2+\cdots+k_r})}
     q_{n_1}(x;a_1q^{k_2+\cdots+k_r}|q), \]
and if we use the sum \eqref{qLag} and the expression \eqref{Cr1} for $C_{n_1}(a_1q^{k_2+\cdots+k_r})$ then after some calculus we find the desired expression \eqref{mqLag}, provided  
\[    C_{\vec{n}}(\vec{a}) = \frac{q^{\binom{n_1}{2}} a_1^{n_1} (1-q)^{n_1}}{(a_1q;q)_{n_1}} C_{\vec{n}-n_1\vec{e}_1} (\vec{a}-a_1\vec{e}_1). \]
This is achieved by taking $C_{\vec{n}}(\vec{a})$ as in \eqref{Cn}.
\end{proof}

\subsection{Little $q$-Jacobi polynomials}   \label{sec:3.2}
Multiple little $q$-Jacobi polynomials were introduced in \cite{PoVA}, where two kinds were given. Here we only deal with the multiple
little $q$-Jacobi polynomials of the first kind and we will use a different normalization. Recall that the little $q$-Jacobi polynomials
are given by
\begin{equation}   \label{qJac}
    q_n(x;a,b|q) = {}_2\phi_1 \left( \left. \begin{array}{c} q^{-n} , abq^{n+1} \\ aq \end{array} \right| q, qx \right) =
   \sum_{k=0}^n \frac{(q^{-n};q)_k (abq^{n+1};q)_k}{(q;q)_k (aq;q)_k} q^kx^k,
\end{equation} 
where $a=q^\alpha$ and $b = q^\beta$. They are given by the Rodrigues formula
\begin{equation}   \label{qJacRod}
    \frac{(qx;q)_\infty}{(bqx;q)_\infty} x^{\alpha} q_n(x;a,b|q) = C_n(a,b)  D_p^n x^{n+\alpha} \frac{(qx;q)_\infty}{(bq^{n+1}x;q)_\infty} , 
\end{equation}
where
\begin{equation}  \label{Cnab}
   C_n(a,b) = C_n(a) = \frac{q^{\binom{n}{2}}a^n(1-q)^n}{(aq;q)_n}.
\end{equation}
These little $q$-Jacobi polynomials have the orthogonality relations
\[     \sum_{k=0}^\infty  q_n(q^k;a,b|q) q_m(q^k;a,b|q) q^k \frac{(bq;q)_k}{(q;q)_k} a^k = 0, \qquad n \neq m. \]
Observe that for $b=0$ we retrieve the little $q$-Laguerre polynomials.

Multiple little $q$-Jacobi polynomials (of the first kind) $q_{\vec{n}}(x;\vec{a},b|q)$ are obtained by changing the parameter $a$ to a vector 
$\vec{a}=(a_1,\ldots,a_r)$.
If $\vec{n} = (n_1,n_2,\ldots,n_r)$ the orthogonality relations then become
\begin{equation}   \label{mqJacorth}
   \sum_{k=0}^\infty q_{\vec{n}}(q^k;\vec{a},b|q) q^{k\ell} q^k \frac{(bq;q)_k}{(q;q)_k} a_j^k = 0, \qquad 0 \leq  \ell \leq n_j-1, 
\end{equation}
for $1 \leq j \leq r$. One needs the condition $\alpha_i-\alpha_j \not\in \mathbb{Z}$ in order that these orthogonality relations
determine the multiple orthogonal polynomials in a unique way. In \cite[Thm. 2.2]{PoVA} a Rodrigues formula was given
\begin{equation}   \label{mqJacRod}
   \frac{(qx;q)_\infty}{(bqx;q)_\infty} q_{\vec{n}}(x;\vec{a},b|q)
   = C_{\vec{n}}(\vec{a},b) \prod_{j=1}^r \left( x^{-\alpha_j} D_p^{n_j} x^{n_j+\alpha_j} \right)  \frac{(qx;q)_\infty}{(bq^{|\vec{n}|+1}x;q)_\infty},
\end{equation}
where $C_{\vec{n}}(\vec{a},b)$ is a constant that determines the normalization. An expression in terms of a generalized basic hypergeometric
function was given in \cite[Eq. (2.7)]{PoVA}
\[   \frac{(qx;q)_\infty}{(bqx;q)_\infty} q_{\vec{n}}(x;\vec{a},b|q) = {}_{r+1}\phi_r \left( 
    \left. \begin{array}{c}   q^{-|\vec{n}|}/b, a_1q^{n_1+1}, \ldots, a_rq^{n_r+1} \\
                                  a_1q, a_2 q, \ldots , a_rq   \end{array}  \right| q , bqx \right) .  \]
Here we used the normalization so that $q_{\vec{n}}(0;\vec{a},b|q) = 1$, which is different from the normalization in \cite{PoVA}.
Note that the limit $b \to 0$ gives the multiple little $q$-Laguerre polynomials $q_{\vec{n}}(x;\vec{a}|q)$, so that we get the
generalized hypergeometric representation
\[    (qx;q)_\infty q_{\vec{n}}(x;\vec{a}|q) = {}_r\phi_r \left( 
    \left. \begin{array}{c}   a_1q^{n_1+1}, \ldots, a_rq^{n_r+1} \\
                               a_1q, a_2 q, \ldots , a_rq   \end{array}  \right| q , q^{-|\vec{n}|+1}x \right) .  \]

An expression for the multiple little $q$-Jacobi polynomials in terms of a finite $r$-fold sum is given by
\begin{theorem}   \label{thm:mqJac}
If we take the normalizing factor in \eqref{mqJacRod} as $C_{\vec{n}}(\vec{a};b)=C_{\vec{n}}(a)$ as in \eqref{Cn},
then an explicit expression for the multiple little $q$-Jacobi polynomials (of the first kind) is given by
\begin{multline} \label{mqJac}
   q_{\vec{n}}(x;\vec{a},b|q) 
  = \sum_{k_1=0}^{n_1} \cdots \sum_{k_r=0}^{n_r}  \prod_{j=1}^r \frac{(q^{-n_j};q)_{k_j}(a_jbq^{\sum_{i=1}^j n_i+1};q)_{\sum_{i=j}^r k_i}}
                  {(q;q)_{k_j}(a_jq;q)_{\sum_{i=j}^r k_i}} \\
   \times  \prod_{j=1}^{r-1} \frac{(a_jq^{n_j+1};q)_{k_{j+1}+\cdots+k_r}}{(a_jbq^{\sum_{i=1}^j n_i+1};q)_{k_{j+1}+\cdots+k_r}}
       \frac{(qx)^{|\vec{k}|}}{q^{\sum_{j=1}^r n_j \sum_{i=j+1}^r k_i}} .
\end{multline}
These multiple little $q$-Jacobi polynomials are normalized so that $q_{\vec{n}}(0;\vec{a},b|q)=1$.
\end{theorem}
Note that for $b=0$ we get the expression \eqref{mqLag} in Theorem \ref{thm:3.1}.

\begin{proof}
The proof is again by induction on $r$. For $r=1$ one has the usual little $q$-Jacobi polynomials in \eqref{qJac} with $C_n(a,b)$ given
in \eqref{Cnab}. Observe that this normalizing factor is independent of $b$.

Suppose that the result holds for $r-1$. Then the Rodrigues formula \eqref{mqJacRod} gives
\begin{eqnarray*}
  \frac{(qx;q)_\infty}{(bqx;q)_\infty} q_{\vec{n}}(x;\vec{a},b|q) & = & C_{\vec{n}}(\vec{a},b) x^{-\alpha_1} D_p^{n_1} x^{n_1+\alpha_1}
   \prod_{j=2}^r  x^{-\alpha_j} D_p^{n_j} x^{n_j+\alpha_j} \frac{(qx;q)_\infty}{(bq^{|\vec{n}|+1};q)_\infty} \\
  & = & \frac{C_{\vec{n}}(\vec{a},b)}{C_{\vec{n}-n_1\vec{e}_1}(\vec{a}-a_1\vec{a},b)}  x^{-\alpha_1} D_p^{n_1} x^{n_1+\alpha_1}
     q_{\vec{n}-n_1\vec{e}_1}(x;\vec{a}-a_1\vec{e}_1;bq^{n_1}|q),
\end{eqnarray*}
where we used the Rodrigues formula \eqref{mqJacRod} with $r-1$ for the product $\prod_{j=2}^r$. Now use the induction hypothesis
to express $q_{\vec{n}-n_1\vec{e}_1}(x;\vec{a}-a_1\vec{e}_1,bq^{n_1}|q)$ as an $(r-1)$-fold sum to find
\begin{multline*}   q_{\vec{n}}(x;\vec{a},b|q) = \frac{C_{\vec{n}}(\vec{a},b)}{C_{\vec{n}-n_1\vec{e}_1}(\vec{a}-a_1\vec{a},b)} 
   \sum_{k_2=0}^{n_2} \cdots \sum_{k_r=0}^{n_r} \prod_{j=2}^r \frac{(q^{-n_j};q)_{k_j}(a_jbq^{\sum_{i=1}^j n_i + 1};q)_{\sum_{i=j}^r k_i}}
         {(q;q)_{k_j} (aq;q)_{\sum_{i=j}^r k_i}} \\
        \times\ \prod_{j=2}^{r-1} \frac{(a_jq^{n_j+1};q)_{\sum_{i=j+1}^r k_i}}{(a_jbq^{\sum_{i=1}^j n_i+1};q)_{\sum_{i=j+1}^r k_i}}
         \frac{q^{k_2+\cdots+k_r}}{q^{\sum_{j=2}^r n_j \sum_{i=j+1}^r k_i}} x^{-\alpha_1} D_p^{n_1} x^{n_1+\alpha_1+k_2+\cdots+k_r}
     \frac{(qx;q)_\infty}{(bq^{n_1+1};q)_\infty} .
\end{multline*}
Now use the Rodrigues formula \eqref{qJacRod} for $r=1$ to find
\begin{multline*}
    D_p^{n_1} x^{n_1+\alpha_1+k_2+\cdots+k_r}    \frac{(qx;q)_\infty}{(bq^{n_1+1};q)_\infty} \\
     = \frac{1}{C_{n_1}(a_1q^{k_2+\cdots+k_r},bq^{n_1})} \frac{(qx;q)_\infty}{(bqx;q)_\infty} x^{\alpha_1+k_2+\cdots+k_r} 
q_{n_1}(x;aq^{k_2+\cdots+k_r},b|q). 
\end{multline*}
Then use the sum \eqref{qJac} and \eqref{Cnab} to find  
\begin{multline*}
  q_{\vec{n}}(x;\vec{a},b|q)  = \frac{C_{\vec{n}}(\vec{a},b)}{q^{\binom{n_1}{2}}a_1^n(1-q)^n C_{\vec{n}-n_1\vec{e}_1}(\vec{a}-a_1\vec{a},b)} \\
  \times  \sum_{k_1=0}^{n_1}\cdots \sum_{k_r=0}^{n_r} \prod_{j=1}^r \frac{(q^{-n_j};q)_{k_j}}{(q;q)_{k_j}}  
   \prod_{j=2}^r \frac{(a_jbq^{\sum_{i=1}^j n_i + 1};q)_{\sum_{i=j}^r k_i}}{(a_jq;q)_{\sum_{i=j}^r k_i}} \\
    \times \prod_{j=2}^{r-1} \frac{(a_jq^{n_j+1};q)_{\sum_{i=j+1}^r k_i}}{(a_jbq^{\sum_{i=1}^j n_i+1};q)_{\sum_{i=j+1}^r k_i}}
    \frac{(a_1q^{k_2+\cdots+k_r+1};q)_{n_1} (a_1bq^{k_2+\cdots+k_r+n_1+1};q)_{k_1}}{(a_1q^{k_2+\cdots+k_r+1};q)_{k_1}}      
 \frac{(qx)^{|\vec{k}|}}{q^{\sum_{j=1}^r n_j \sum_{i=j+1}^r k_i}} .
\end{multline*}
Now use
\begin{eqnarray*}
  (a_1q^{k_2+\cdots+k_r};q)_{n_1} &=& \frac{(a_1q;q)_{n_1} (a_1q^{n_1+1};q)_{k_2+\cdots+k_r}}{(a_1q;q)_{k_2+\cdots+k_r}}, \\
  (a_1q^{k_2+\cdots+k_r};q)_{k_1} &=& \frac{a_1q;q)_{k_1+\cdots+k_r}}{(a_1q;q)_{k_2+\cdots+k_r}}, \\
  (a_1bq^{k_2+\cdots+k_r+n_1+1};q)_{k_1} &=& \frac{a_1bq^{n_1+1};q)_{k_1+\cdots+k_r}}{(a_1bq^{n_1+1};q)_{k_2+\cdots+k_r}}
\end{eqnarray*}
to find the $r$-fold sum in \eqref{mqJac}. The normalization $q_{\vec{n}}(0;\vec{a},b|q)=1$ is obtained when
\[   \frac{C_{\vec{n}}(\vec{a},b)(a_1q;q)_{n_1}}{q^{\binom{n_1}{2}}a_1^n(1-q)^n C_{\vec{n}-n_1\vec{e}_1}(\vec{a}-a_1\vec{a},b)} = 1, \]
which holds for $C_{\vec{n}}(\vec{a},b) = C_{\vec{n}}(\vec{a})$ in \eqref{Cn}. Observe that this factor does not depend on~$b$. 
\end{proof}

\subsection{Multiple Al-Salam--Chihara polynomials}    \label{sec:3.3}
In this section we will take $r$ weights $(w_1,w_2,\ldots,w_r)$ on $[-1,1]$ with $w_r(\theta) = w(\theta;a_j,b|q)$, where
$w(\theta;a,b|q)$ is the Al-Salam--Chihara weight given in \eqref{AlSalChiw} and
$\vec{a}=(a_1,\ldots,a_r)$ such that $a_j = q^{\alpha_j}$ and $\alpha_i - \alpha_j \not\in \mathbb{Z}$ whenever $i \neq j$. As usual with the
Al-Salam--Chihara weight, we take $|a_j| < 1$ and $|b|<1$. The corresponding multiple orthogonal polynomials can then be obtained
from the little $q$-Laguerre polynomials by using Proposition \ref{prop:3.3} and the transformation $T_b$ given in Theorem \ref{thm:2.2}.
In fact Theorem \ref{thm:2.1} can be extended to multiple orthogonal polynomials as follows.

\begin{theorem}   \label{thm:3.3}
The multiple Al-Salam--Chihara polynomials $p_{\vec{n}}(x;\vec{a},b|q)$ for the weights 
$w(\theta;a_1,b|q),w(\theta;a_2,b|q),\ldots,w(\theta;a_r,b|q)$, with $w(\theta;a,b|q)$ given in \eqref{AlSalChiw}, are given by
$p_{\vec{n}}(x;\vec{a},b|q) = T_b q_{\vec{n}}(x;\frac{b}{q}\vec{a}|q)$, where $T_b$ is the linear transformation given in Theorem \ref{thm:2.2}
and $q_{\vec{n}}(x;\frac{b}{q}\vec{a}|q)$ are the multiple little $q$-Laguerre polynomials given in Section \ref{sec:3.1}.
An explicit expression is given by
\begin{multline*}
 p_{\vec{n}}(x;\vec{a},b|q) = \sum_{k_1=0}^{n_1} \cdots \sum_{k_r=0}^{n_r} 
 \prod_{j=1}^r \frac{(q^{-n_j};q)_{k_j}}{(q;q)_{k_j}(a_jb;q)_{k_j}}  
 \prod_{j=1}^{r-1} \frac{(a_jbq^{n_j};q)_{k_{j+1}+\cdots+k_r}}{(a_jbq^{k_j};q)_{k_{j+1}+\cdots+k_r}} \\
     \times  \frac{q^{|\vec{k}|}}{q^{\sum_{j=1}^r n_j \sum_{i=j+1}^r k_i}} (be^{i\theta},be^{-i\theta};q)_{|\vec{k}|} .
\end{multline*}
\end{theorem}

\begin{proof}
Let $T_b$ be the linear tarnsformation that acts on polynomials like
\[   T_b x^k = (be^{i\theta},be^{-i\theta};q)_k,  \]
then we apply this to the multiple little $q$-Laguerre polynomials with parameters 
$\frac{b}{q} \vec{a} = \left( \displaystyle \frac{a_1b}{q}, \ldots, \frac{a_rb}{q} \right)$ to find $T_b q_{\vec{n}}(x;\frac{b}{q} \vec{a}|q)$.
If we use Proposition \ref{prop:3.3} then
\[  \langle T_{a_j} x^\ell , T_b q_{\vec{n}}(x;\frac{b}{q}\vec{a}|q) \rangle_{\textup{cont}(a_j,b)}
     = \langle x^\ell , q_{\vec{n}}(x;\frac{b}{q} \vec{a} |q) \rangle_{\textup{dis}(a_jb)}, \]
which is
\begin{multline*}
  \frac{1}{2\pi} \int_{-1}^1 (a_je^{i\theta},a_je^{-i\theta};q)_\ell T_b q_{\vec{n}}(x;\frac{b}{q}\vec{a}|q) w(\theta;a_j,b|q)\, \frac{dx}{\sqrt{1-x^2}}
    \\
  = \frac{1}{(q;q)_\infty} \sum_{k=0}^\infty q^{k\ell} q_{\vec{n}}(q^k;\frac{b}{q} \vec{a}|q) \frac{(a_jb)^k}{(q;q)_k}. 
\end{multline*}
The latter sum is $0$ because of \eqref{mqLagorth} when $0 \leq \ell \leq n_j-1$ for $1 \leq j \leq r$. This shows that
$T_b q_{\vec{n}}(x;\frac{b}{q} \vec{a}|q)$ satisfies the multiple orthogonality conditions with respect to the $r$ Al-Salam--Chihara weights
$w(\theta;a_1,b|q),w(\theta;a_2,b|q),\ldots,w(\theta;a_r,b|q)$.
The explicit expression as an $r$-fold sum is obtained by applying $T_b$ to the sum \eqref{mqLag} in Theorem \ref{thm:3.1} for the parameters $ba_j/q$.
\end{proof} 

\subsection{Multiple continuous dual $q$-Hahn polynomials}    \label{sec:3.4}
We now take $r$ weights $(w_1,\ldots,w_r)$ on $[-1,1]$ by using the continuous dual $q$-Hahn weight $w(\theta;a_j,b,c|q)$ of \eqref{qHahnw}
with $r$ different parameters $\vec{a}=(a_1,\ldots,a_r)$, keeping $b,c$ fixed. Again we let $a_j=q^{\alpha_j}$ and assume that 
$\alpha_i - \alpha_j \not\in \mathbb{Z}$ whenever $i \neq j$. This ensures that the multiple orthogonality conditions
\[  \int_{-1}^1 p_{\vec{n}}(x;\vec{a},b,c|q) x^\ell w(\theta;a_j,b,c|q)\, \frac{dx}{\sqrt{1-x^2}} = 0, \qquad 0 \leq \ell \leq n_j-1, \]
for $1 \leq j \leq r$, give $|\vec{n}|$ equations that determine the $p_{\vec{n}}(x;\vec{a},b,c|q)$ uniquely (up to a multiplicative factor).
These multiple continuous dual $q$-Hahn polynomials can be obtained by using the linear transformation $T_{b,c}$ given in Theorem \ref{thm:2.8}
to the multiple little $q$-Laguerre polynomials. The extension of Theorem \ref{thm:2.7} to multiple continuous dual $q$-Hahn polynomials is:

\begin{theorem}   \label{thm:3.4}
The multiple continuous dual $q$-Hahn polynomials $p_{\vec{n}}(x;\vec{a},b,c|q)$ for the weights 
$w(\theta;a_1,b,c|q),w(\theta;a_2,b,c|q),\ldots,w(\theta;a_r,b,c|q)$, with $w(\theta;a,b,c|q)$ given in \eqref{qHahnw}, are given by
$p_{\vec{n}}(x;\vec{a},b,c|q) = T_{b,c} q_{\vec{n}}(x;\frac{b}{q}\vec{a}|q)$, where $T_{b,c}$ is the linear transformation given in Theorem \ref{thm:2.8}
and $q_{\vec{n}}(x;\frac{b}{q}\vec{a}|q)$ are the multiple little $q$-Laguerre polynomials given in Section \ref{sec:3.1}.
An explicit expression is given by
\begin{multline*}
 p_{\vec{n}}(x;\vec{a},b,c|q) = \sum_{k_1=0}^{n_1} \cdots \sum_{k_r=0}^{n_r} 
 \prod_{j=1}^r \frac{(q^{-n_j};q)_{k_j}}{(q;q)_{k_j}(a_jb;q)_{k_j}}  
 \prod_{j=1}^{r-1} \frac{(a_jbq^{n_j};q)_{k_{j+1}+\cdots+k_r}}{(a_jbq^{k_j};q)_{k_{j+1}+\cdots+k_r}} \\
     \times  \frac{q^{|\vec{k}|}}{q^{\sum_{j=1}^r n_j \sum_{i=j+1}^r k_i}} \frac{(be^{i\theta},be^{-i\theta};q)_{|\vec{k}|}}{(bc;q)_{|\vec{k}|}} .
\end{multline*}
\end{theorem}

\begin{proof}
If we use Proposition \ref{prop:2.9} then
\[  \langle T_{a_j,c} x^\ell , T_{b,c} q_{\vec{n}}(x;\frac{b}{q} \vec{a} |q) \rangle_{\textup{cont}(a_j,b,c)} 
 = \langle x^\ell, q_{\vec{n}}(x;\frac{b}{q} \vec{a}|q) \rangle_{\textup{dis}(a_jb)}, \]
which is
\begin{multline*}
  \frac{(a_jc,bc;q)_\infty}{2\pi} \int_{-1}^1 \frac{(a_je^{i\theta},a_je^{-i\theta};q)_\ell}{(a_jc;q)_k} T_{b,c} q_{\vec{n}}(x;\frac{b}{q} \vec{a}|q)
   w(\theta;a_j,b,c|q) \, \frac{dx}{\sqrt{1-x^2}} \\
   = \frac{1}{(q;q)_\infty} \sum_{k=0}^\infty  q^{k\ell} q_{\vec{n}}(q^k;\frac{b}{q} \vec{a}|q) \frac{(a_jb)^k}{(q;q)_k}
\end{multline*}
and the latter sum is $0$ whenever $0 \leq \ell \leq n_j-1$ because of the multiple orthogonality conditions \eqref{mqLagorth} for the multiple little
$q$-Laguerre polynomials. This shows that $T_{b,c} q_{\vec{n}}(x;\frac{b}{q} \vec{a}|q)$ satisfies the multiple orthogonality conditions
with respect to the $r$ continuous dual $q$-Hahn weights $w(\theta;a_1,b,c|q), \ldots, w(\theta;a_r,b,c|q)$. The $r$-fold sum
is obtained by applying $T_{b,c}$ to the $r$-fold sum \eqref{mqLag} for the multiple little $q$-Laguerre polynomials.
\end{proof}
Observe that for $c=0$ one has the multiple Al-Salam--Chihara polynomials given in Theorem \ref{thm:3.3}. 

\subsection{Multiple Askey--Wilson polynomials}     \label{sec:3.5}
Finally we will obtain the multiple Askey--Wilson polynomials by extending Theorem \ref{thm:2.4}. We choose the $r$ weights
$(w_1,\ldots,w_r)$ on $[-1,1]$ by taking the Askey--Wilson weights $w(\theta;a_j,b,c,d|q)$ of \eqref{AWw} with $r$ different parameters 
$\vec{a}=(a_1,\ldots,a_r)$, keeping $b,c,d$ fixed. Of course we could have taken $r$ different parameters $(b_1,\ldots,b_r)$, keeping
$a,c,d$ fixed or $r$ different $c$ or $d$ parameters, but since the Askey--Wilson weight is symmetric in $(a,b,c,d)$ this would not give anything new.
Again we let $a_j=q^{\alpha_j}$ and assume that $\alpha_i - \alpha_j \not\in \mathbb{Z}$ whenever $i \neq j$ to ensure that the multiple orthogonality conditions
\[  \int_{-1}^1 p_{\vec{n}}(x;\vec{a},b,c,d|q) x^\ell w(\theta;a_j,b,c,d|q)\, \frac{dx}{\sqrt{1-x^2}} = 0, \qquad 0 \leq \ell \leq n_j-1, \]
for $1 \leq j \leq r$, give $|\vec{n}|$ equations that determine the $p_{\vec{n}}(x;\vec{a},b,c|q)$ uniquely (up to a multiplicative factor).

\begin{theorem}   \label{thm:3.5}
The multiple Askey--Wilson polynomials $p_{\vec{n}}(x;\vec{a},b,c,d|q)$ for the weights 
$w(\theta;a_1,b,c,d|q),w(\theta;a_2,b,c,d|q),\ldots,w(\theta;a_r,b,c,d|q)$, with $w(\theta;a,b,c,d|q)$ given in \eqref{AWw}, are given by
$p_{\vec{n}}(x;\vec{a},b,c,d|q) = T_{b,c,d} q_{\vec{n}}(x;\frac{b}{q}\vec{a},\frac{cd}{q}|q)$, where $T_{b,c,d}$ is the linear transformation given in Theorem \ref{thm:2.5} and $q_{\vec{n}}(x;\frac{b}{q}\vec{a},\frac{cd}{q}|q)$ are the multiple little $q$-Jacobi polynomials given in Section \ref{sec:3.2}.
An explicit expression is given by
\begin{multline*}
   p_{\vec{n}}(x;\vec{a},b,c,d|q) 
  = \sum_{k_1=0}^{n_1} \cdots \sum_{k_r=0}^{n_r}  \prod_{j=1}^r \frac{(q^{-n_j};q)_{k_j}(a_jbcdq^{\sum_{i=1}^j n_i-1};q)_{\sum_{i=j}^r k_i}}
                  {(q;q)_{k_j}(a_jb;q)_{\sum_{i=j}^r k_i}} \\
   \times  \prod_{j=1}^{r-1} \frac{(a_jbq^{n_j};q)_{k_{j+1}+\cdots+k_r}}{(a_jbcdq^{\sum_{i=1}^j n_i-1};q)_{k_{j+1}+\cdots+k_r}}
       \frac{q^{|\vec{k}|}}{q^{\sum_{j=1}^r n_j \sum_{i=j+1}^r k_i}} \frac{(be^{i\theta},be^{-i\theta};q)_k}{(bc;q)_k(bd;q)_k} .
\end{multline*}
\end{theorem}

\begin{proof}
Proposition \ref{prop:2.6} gives
\[  \langle T_{a_j,c,d} x^\ell , T_{b,c,d} q_{\vec{n}}(x;\frac{b}{q} \vec{a},\frac{cd}{q} |q) \rangle_{\textup{cont}(a_j,b,c,d)} 
 = \langle x^\ell, q_{\vec{n}}(x;\frac{b}{q} \vec{a},\frac{cd}{q}|q) \rangle_{\textup{dis}(a_jb,cd)}, \]
which is
\begin{multline*}
  \frac{(a_jc,a_jd,bc,bd,cd;q)_\infty}{2\pi} \int_{-1}^1 \frac{(a_je^{i\theta},a_je^{-i\theta};q)_\ell}{(a_jc,a_jd;q)_k} T_{b,c,d} 
   q_{\vec{n}}(x;\frac{b}{q}\vec{a},\frac{cd}{q}|q)
   w(\theta;a_j,b,c,d|q) \, \frac{dx}{\sqrt{1-x^2}} \\
   = \frac{1}{(q;q)_\infty} \sum_{k=0}^\infty  q^{k\ell} q_{\vec{n}}(q^k;\frac{b}{q}\vec{a},\frac{cd}{q}|q) (a_jb)^k \frac{(cd;q)_k}{(q;q)_k}
\end{multline*}
and the latter sum is $0$ whenever $0 \leq \ell \leq n_j-1$ because of the multiple orthogonality condition \eqref{mqJacorth} for the multiple little
$q$-Jacobi polynomials. This shows that $T_{b,c,d} q_{\vec{n}}(x;\frac{b}{q}\vec{a},\frac{cd}{q}|q)$ satisfies the multiple orthogonality conditions
with respect to the $r$ continuous dual Askey--Wilson weights $w(\theta;a_1,b,c,d|q), \ldots, w(\theta;a_r,b,c,d|q)$. The $r$-fold sum
is obtained by applying $T_{b,c,d}$ to the $r$-fold sum \eqref{mqJac} for the multiple little $q$-Jacobi polynomials, where we first
replace $b$ by $cd/q$ and then $a_j$ by $ba_j/q$.
\end{proof}
Observe that for $d=0$ one has the multiple continuous dual $q$-Hahn polynomials given in Theorem \ref{thm:3.4}, and for $c=d=0$ the multiple Al-Salam--Chihara polynomials given in Theorem \ref{thm:3.3}. 

\section{Concluding remarks}
We have given a number of new families of multiple orthogonal polynomials starting from the discrete multiple little $q$-Jacobi polynomials and 
working our way up to the multiple Askey--Wilson polynomials, with limiting cases the multiple continuous dual $q$-Hahn polynomials and
multiple Al-Salam--Chihara polynomials. These three new families are multiple orthogonal polynomials for an AT-system whenever 
$\alpha_i-\alpha_j \not\in  \mathbb{Z}$, see \cite[\S 23.1.2]{Ismail} or \cite[Ch. 4, \S 4]{NikiSor} for more on AT-systems.
For $r=2$ we see that the ratio $w_2/w_1$ of the two weights for multiple Askey--Wilson polynomials is
\[   \frac{w(\theta;a_2,b,c,d|q)}{w(\theta;a_1,b,c,d|q)} = \frac{(a_1e^{i\theta},a_1e^{-i\theta};q)_\infty}{(a_2e^{i\theta},a_2e^{-i\theta};q)_\infty}
    = \prod_{k=0}^\infty \frac{1+a_1^2q^{2k}-2a_1q^kx}{1+a_2^2q^{2k}-2a_2q^kx}, \]
which is a meromorphic function with poles at the points $x_k = (1+a_2^2q^{2k})/2a_2q^k$ and zeros at the points $y_k = (1+a_1^2q^{2k})/2a_1q^k$.
When $a_1,a_2>0$ these poles and zeros are in $(1,\infty)$. The multiple Askey--Wilson polynomials therefore behave very much like a Nikishin
system for which
\[  \frac{w_2(x)}{w_1(x)} = \int_a^b \frac{d\sigma(t)}{x-t}, \]
with $\sigma$ a positive measure and $[a,b] \cap [-1,1] = \emptyset$, in particular with $\sigma$ a discrete
measure supported on the poles $\{x_k, k=0,1,2,\ldots \}$, see e.g., \cite{AptLM}. The weights for multiple continuous dual $q$-Hahn polynomials
and multiple Al-Salam--Chihara polynomials have the same ratio, and hence the same meromorphic function. Hence these multiple orthogonal
polynomials may give some insight into the behavior of multiple orthogonal polynomials for a Nikishin system.
Various other properties of our new multiple orthogonal polynomials would be of interest as well, such as the nearest neighbor recurrence
relations and the Rodrigues formula.

\section*{Acknowledgements}
J.P. Nuwacu is supported by VLIR-UOS project Belgium/Burundi CUI UB ZIUS2018AP022 and W. Van Assche by FWO research project G.0C9819N 
and EOS project PRIMA 30889451.

\bigskip\bigskip

\noindent\parbox{3in}{Jean Paul Nuwacu \\ 
D\'epartement de Math\'ematiques \\
Universit\'e du Burundi \\
BP 2700 Bujumbura\\
Burundi \\
\texttt{jean-paul.nuwacu@ub.edu.bi}}
\bigskip

\noindent\parbox{3in}{Walter Van Assche \\
Department of Mathematics \\
KU Leuven \\
Celestijnenlaan 200B box 2400 \\ 
BE-3001 Leuven \\
Belgium \\
\texttt{walter.vanassche@kuleuven.be}}

\end{document}